\documentclass[a4paper]{article}
\usepackage{amsfonts}
\usepackage{amsmath}
\usepackage{amsthm}
\usepackage{amssymb}
\usepackage{hyperref}
\usepackage[utf8]{inputenc}
\usepackage[capitalize,noabbrev]{cleveref}
\usepackage{autonum}
\usepackage[T1]{fontenc}
\usepackage{mathrsfs}
\usepackage{thmtools}
\usepackage{subcaption}
\usepackage{tikz}
\usetikzlibrary[positioning,decorations.pathmorphing]

\usepackage{tikz-cd}
\usepackage[left=2.5cm,top=3cm,right=2.5cm,bottom=3cm,bindingoffset=0.5cm]{geometry}

\newtheorem{theorem}{Theorem}
\newtheorem*{theorem*}{Theorem}
\newtheorem{corollary}[theorem]{Corollary}
\newtheorem{hypothesis}[theorem]{Hypothesis}
\newtheorem{proposition}[theorem]{Proposition}
\newtheorem{lemma}[theorem]{Lemma}

\theoremstyle{definition}

\theoremstyle{remark}
\newtheorem{remark}[theorem]{Remark}

\DeclareMathOperator{\rad}{r}
\DeclareMathOperator{\SL}{SL}

\DeclareMathOperator{\ind}{\mathbf{1}}
\DeclareMathOperator{\vol}{vol}

\DeclareMathOperator{\covol}{covol}
\DeclareMathOperator{\Res}{Res}
\DeclareMathOperator{\Gal}{Gal}
\DeclareMathOperator{\card}{\#}

\DeclareMathOperator{\Tr}{tr}
\DeclareMathOperator{\N}{N}

\DeclareMathOperator{\Ar}{\widetilde{Cl}}
\DeclareMathOperator{\Cl}{Cl}
\newcommand{\OK}{\mathcal{O}_K}
\newcommand{\Eps}{\mathscr{E}}

\newcommand{\KR}{K_{\mathbb{R}}}

\newcommand{\mell}[1]{\left[ \mathcal{M} #1 \right] }
\newcommand{\mellK}[1]{\left[ \mathcal{M}_K #1 \right] }
\newcommand*\diff{\,\mathop{}\! \mathrm{d}\hspace*{-0.15em}\bar{}\hspace*{0.13em}}
\newcommand*\diffx{\,\mathop{}\! \mathrm{d}\hspace*{-0.15em}\bar{}^{\times}\hspace*{-0.27em}}

\title{Mean Value for Random Ideal Lattices}
\author{Nihar Gargava, Maryna Viazovska}

\begin{document}
\maketitle

\begin{abstract}
We investigate the average number of lattice points within a ball for the $n$th cyclotomic number field, where the lattice is chosen at random from the set of unit determinant ideal lattices of the field. We show that this average is nearly identical to the average number of lattice points in a ball among all unit determinant random lattices of the same dimension. To establish this result, we apply the Hecke integration formula and subconvexity bounds on Dedekind zeta functions of cyclotomic fields. 

The symmetries arising from the roots of unity in an ideal lattice allow us to prove the existence of ideal lattice packings of dimension $\varphi(n)$ and density $ n \cdot 2^{- \varphi(n)}  (1+o(1))$ as $n$ goes to infinity.
\end{abstract}

\section*{Introduction}

The following is a classical fact due to Siegel \cite{Sie45} that is often referred to as the Siegel mean value theorem. If $B \subseteq \mathbb{R}^{d}$ is a ball of volume $V$, then for any $d \geq 2$
\begin{equation}
	\mathbb{E} (\card B \cap \Lambda) = 1 + V, {\text{ for a Haar-random unit covolume lattice }\Lambda \in \SL_{d}(\mathbb{R})/\SL_{d}(\mathbb{Z})}.
	\label{eq:siegel}
\end{equation}
The $1$ on the right  side is the contribution of $0 \in \Lambda$, which lies in each lattice. It is worth noticing here that the expected value does not depend on $d$. This means that as long as the volume of the ball is fixed, the expected number of lattice points in a ball cannot tell us which dimension our lattices live in.
The goal of this paper is to have a version of~\cref{eq:siegel} 
for an ensemble of special lattices that have an arithmetic structure. Let us first set up some notation.

Let $K$ be a number field of degree $d$ and $\OK$ be the ring of integers $K$.
We define an ideal lattice to be a 
$\OK$-module in $K \otimes \mathbb{R} \simeq \mathbb{R}^{d}$, normalized to have covolume 1.
On $K \otimes \mathbb{R}$, one can induce a quadratic form naturally using the trace of the number field $K$ (see Equation (\ref{eq:norm})) and hence one can talk about unit covolume ideal lattices. 
The collection of all such lattices is a compact abelian Lie group denoted as $\Ar(K)$, called the Arakelov class group \cite{BDPW20} of the number field $K$. 
Naturally, we can embed $\Ar(K) \subseteq \SL_d(\mathbb{R})/\SL_d(\mathbb{Z})$.
One can then try to compute $\mathbb{E}_{\Lambda \in \Ar(K)} ( \card B \cap \Lambda)$ for a number field $K$.
See~\cref{fig:ideals} for a visualization for the case of $K=\mathbb{Q}[\sqrt{5}]$.

\begin{figure}
	\centering
\begin{subfigure}[t]{0.45\linewidth}
	\centering
\begin{tikzpicture}

    \draw[step=0.50,gray,very thin] (-4,-2) grid (0,2);

    \draw[->,gray,thick] (-4,0) -- (0,0) ;
    \draw[->,gray,thick] (-2,-2) -- (-2,2) ;

    \fill[gray] (-2,0) circle (2pt);

    \foreach \i in {-6,-5,-4,-3,-2,-1,0,1,2,3,4,5,6} {
        \foreach \j in {-6,-5,-4,-3,-2,-1,0,1,2,3,4,5,6} {
            \pgfmathsetmacro{\x}{0.5*\i + 0.5*1.618*\j-2} 
            \pgfmathsetmacro{\y}{0.5*\i - 0.5*0.618*\j} 
            \ifdim\x pt>-4.01pt \ifdim\x pt<0.01pt
                \ifdim\y pt>-2.01pt \ifdim\y pt<2.01pt
                    \fill[cyan] (\x,\y) circle (1.5pt);
                \fi\fi
            \fi\fi
        }
    }

\end{tikzpicture}
\caption{
For $K = \mathbb{Q}[\sqrt{5}]$, we can plot $\OK =\mathbb{Z}[\tfrac{1+\sqrt{5}}{2} ]
$ as the lattice $ \mathbb{Z} \cdot 
\left( \substack{1 \vspace{0.1cm} \\ 1}  \right)
+ \mathbb{Z} \cdot \left(\substack{ (1+ \sqrt{5})/{2} \\ (1-\sqrt{5})/{2} }  \right)$ 
in 
$K \otimes \mathbb{R} \simeq \mathbb{R}^{2}$.
}
 
\end{subfigure}
 \hspace{1cm}
\begin{subfigure}[t]{0.44\linewidth}
\centering
\begin{tikzpicture}
    \draw[step=0.50,gray,very thin] (-2,-2) grid (2,2);

    \draw[->,gray,thick] (-2,0) -- (2,0) ;
    \draw[->,gray,thick] (0,-2) -- (0,2) ;

    \fill[gray] (0,0) circle (2pt);

\foreach \t in {-4,-3,-2,-1,0,1,2,3}{
    \foreach \i in {-6,-5,-4,-3,-2,-1,0,1,2,3,4,5,6} {
        \foreach \j in {-6,-5,-4,-3,-2,-1,0,1,2,3,4,5,6} {
            \pgfmathsetmacro{\x}{0.5*\i*exp(\t*0.12) + 0.5*1.618*\j*exp(\t*0.12)} 
            \pgfmathsetmacro{\y}{0.5*\i*exp(-\t*0.12) - 0.5*0.618*\j*exp(-\t*0.12)} 
            \ifdim\x pt>-2.01pt \ifdim\x pt<2.01pt
                \ifdim\y pt>-2.01pt \ifdim\y pt<2.01pt
                    \pgfmathsetmacro{\s}{(\t/4+1)*50}
                    \fill[cyan!\s!red!80] (\x,\y) circle (1.5pt);
                \fi\fi
            \fi\fi
        }
    }
}
\end{tikzpicture}

\caption{
	Various ideal lattices for $K = \mathbb{Q}[\sqrt{5}]$
	superimposed on each other. 
}
\end{subfigure}
\caption{Ideal lattices for $K = \mathbb{Q}[\sqrt{5}]$. For real quadratic fields, when plotted this way, the lattice point of ideal lattices always lie in the set $\{(x,y) \in \mathbb{R}^{2} \mid x \cdot y \in \mathbb{Z}\}$.
Note that the ideal lattices in these plots are not with the unit covolume normalization.
}
 \label{fig:ideals}
\end{figure}

The main theorem that we present is the following. 
\begin{theorem}
\label{th:main}
Let $K$ be a cyclotomic number field. For each $K$, choose $B \subseteq K \otimes \mathbb{R}$ to be an origin-centered ball of volume $V$ with respect to the trace form (\ref{eq:norm}). Then, as $\deg K \rightarrow \infty$, we have for some $\eta>0$
\begin{equation}
\mathbb{E}(\card B \cap \Lambda) = 1 + V + \varepsilon(V,K)
\text{ for a Haar-random unit covolume }\Lambda \in \Ar(K).
\label{eq:main_estimate}
\end{equation}
where $\varepsilon(V,K)$ is an error term\footnote{The $\sqrt{V}$ in the error term agrees with the results of S\"odergren-Str\"ombergsson \cite{SS19} 
who show that for a Haar random lattice $\Lambda$ in 
$\SL_{d}(\mathbb{R})/\SL_{d}(\mathbb{Z})$, 
the random variable
\begin{equation}
	\frac{(\card B \cap \Lambda) - ( 1 + V  )}{ \sqrt{2V}} \rightarrow \mathcal{N}(0,1) \text{ in distribution as }d \to \infty,
\end{equation}
assuming that $\log V = o(d)$ as $d \rightarrow \infty$. Here $\mathcal{N}(0,1)$ is a standard Gaussian random variable.}
that satisfies
\begin{equation}
	|\varepsilon(V,K)| \leq c \cdot  \sqrt{V}  \Delta_K^{{-\eta}} .
	\label{eq:error_term}
\end{equation}
for some constant $c > 0$ which is independent of $K$ and $V$ (but may depend on $\eta$).
Here $\Delta_{K}$ is the absolute value of the discriminant of the number field $K$.
\end{theorem}

The discriminant $\Delta_K$
of a cyclotomic field $K$ with $\deg K = d$ 
is at least $e^{c d \log d}$ for some constant $c>0$ (see~\cref{se:cyclo_satisfy}). 
{So actually, if one follows through the proof in Section \ref{se:proof},
a more precise version of the error term in Equation (\ref{eq:error_term}) can be that there exist constants $c_1,c_2>0$ chosen uniformly for all cyclotomic number fields $K$ such that 
$$|\varepsilon(V,K)|\leq V^{1/2}\,\exp\left(-c_1\,{d\log(d)}+c_2\,d \right) \text{ where }d= \deg K.$$
}

The best value of $\eta$ in Theorem \ref{th:main} depends on the availability of subconvexity bounds of the Dedekind zeta functions on the critical line. 
The currently available bounds \cite{PY23} are sufficient to prove any value of $\eta<1/12$. 
{ Under the Generalized Riemann Hypothesis, one can presumably show $\eta$ arbitrarily close to $1/4$. }
It must be noted that using only convexity bounds is not sufficient for showing $\eta >0$.

Furthermore, we do not expect this behaviour of the error term to be easily observed
in small degrees, since in our estimates we demand $\log (\deg K)$ to grow large enough.

\subsection*{Motivation}

\subsubsection*{Sphere packings}

Our main motivation here is that using an argument exactly like Venkatesh \cite{AV}, it is possible to improve the lower bound on the lattice packing constant by a factor of $2$.
It is possible to conclude the following from Theorem \ref{th:main}.
\begin{corollary}
\label{co:improv}
Let 
\begin{equation}
c_d = \sup \{\vol(B) \mid \exists g \in \SL_d(\mathbb{R}), g \mathbb{Z}^{d} \cap B = \{0\} \}.
\end{equation}
Then
for $n$ large enough and some constants $c_{1},c_{2}>0$ independent of $n$
\begin{equation}
	c_{\varphi(n) } \geq n - c_{1} e^{- c_{2} \varphi(n) \log n} .
\end{equation}
In particular, there exist infinitely many dimensions $d$ such that 
\begin{equation}
c_d \geq d \log \log d - 
c_{1} e^{-c_{2} d \log d }.
\end{equation}
\end{corollary}

\begin{proof} {\bf (of Corollary \ref{co:improv})}
Let $K = \mathbb{Q}(\mu_n)$, where $\mu_{n} \in \overline{\mathbb{Q}}$ is a primitive $n$th root of unity. 

Observe that both the ball $B$, due to the definition of the trace form in Equation (\ref{eq:norm}), and the lattice $\Lambda$, being an ideal lattice, are invariant under the action of the roots of unity $\mu_n \subseteq \OK^{\times}$. 
Hence, we obtain that $\card (B \cap \Lambda )\in 1 + n \mathbb{Z}_{\geq 0}$. Hence, if $V < n - O(e^{-c \varphi(n) \log n})$, we know that $B \cap \Lambda = \{0\}$.

Then, using the choice of $n$ as the product of first $k$ primes like in \cite{AV}
and letting $k \rightarrow \infty$, we maximize the ratio of $\varphi(n)$ and $n$ to obtain the sequence of dimensions.
\end{proof}

In general, it is an open problem to exhibit an ``explicitly'' defined unit-covolume lattice $\Lambda \subseteq \mathbb{R}^{d}$
for which $B \cap \Lambda = \{0\}$, where $B \subseteq \mathbb{R}^{d}$ is an origin-centered ball whose volume is at least a constant.
The work of Terras \cite[Corollary 3]{T1980} proves using the non-vanishing of Dedekind zeta functions 
that there are ideal lattices of number fields of arbitrarily large degree 
that satisfy this property for $|\vol(B)| \geq (1- \varepsilon)^{d}$ for $\varepsilon$ arbitrarily small. 
\cref{th:main} and~\cref{co:improv} is a stronger version of Terras' result for the case of cyclotomic number fields with $\vol(B)$ slightly smaller than $d$.

Back when Venkatesh's result had been published, it was not just the best sphere packing bound in high dimensions.
That has now been surpassed by the work of Klartag~\cite{K25}
shows the existence of unit covolume lattices that prove $c_{d} \geq c d^{2}$ for some constant $c > 0$ for all $d \in \mathbb{Z}_{\geq 1}$.

Nonetheless, the case of ideal lattices for sphere packing still remains interesting because of its connections to number theory. 

\subsubsection*{Lattice-based cryptography}

Ideal lattices over cyclotomic fields are of independent interest 
due to their use in post-quantum cryptography \cite{BDPW20}. 
Probabilistic results about behaviour of short vectors in ideal lattices are interesting from this point of view. 
To this end, Theorem \ref{th:main} shows that as long as the average number of lattice points in a ball of fixed volume is concerned, random ideal lattices over cyclotomic fields of large degree 
satisfy similar ``Gaussian heuristics'' as a random unit covolume lattice in $\SL_{d}(\mathbb{R})/\SL_{d}(\mathbb{Z})$. 
This has attracted some interest from the cryptography community. 
For example, a chapter in the doctoral thesis by H. Bambury~\cite{B25} is devoted to explicitly evaluating the error term $\varepsilon(K,V)$ in~\cref{eq:main_estimate}. 

See~\cite{GSVV25} for a similar result about module lattices over cyclotomic fields

\subsection*{General number fields}

Let $\mathcal{S}$ be a set of number fields up to isomorphism. 
Then one can conclude that $\mathcal{S}$ contains finitely many number fields with bounded discriminant. Indeed,
for any $T>0$, it is a known fact that there are finitely many number fields $K$ with $\Delta_{K} \leq T$.

For a set $\mathcal{S}$, the following are the hypotheses that interest us.
\begin{enumerate}

    \item 
\subsubsection*{Subconvexity}

The following is the version of subconvexity hypothesis required for our purpose.
\begin{hypothesis}
	\label{hy:lind}
For some $1/4 \geq \eta_{0}, \eta_{2} > 0$ and some $ \frac{1}{2}  \geq \eta_{1} > 0$, there exists a universal constant $C(\eta_{0},\eta_{1},\eta_{2}) > 0$, 
such that for all number fields $K \in \mathcal{S}$ and all $t \in \mathbb{R}$ one has
\begin{equation}
	\label{eq:lind}
    | \zeta_{K}( \tfrac{1}{2} + it)| \leq C(\eta_{0},\eta_{1},\eta_{2})^{\deg K} \cdot 
    \left( \Delta_{K} \right)^{\frac{1}{4}-\eta_{0}} \cdot
    \left( (1 + |t|)^{r_{1} }  \right)^{ \frac{1}{2}-\eta_{1}} \cdot
    \left( (1 + |t|)^{ 2 r_{2} }  \right)^{ \frac{1}{4}-\eta_{2}}.
\end{equation}
Here $(r_{1},r_{2})$ is the signature of the number field $K$ so that $\deg K = r_{1}  + 2r_{2}$. The constants $\eta_{0},\eta_{1},\eta_{2}$ are required to be independent of $r_{1},r_{2}$ and $K$.
\end{hypothesis}

When $\tfrac{1}{2}-\eta_{1}=\tfrac{1}{4}-\eta_{2} < \frac{1}{4}$, 
this type of bound on $\zeta_{K}(\tfrac{1}{2}+it)$ 
is called a ``subconvexity bound'' with respect to the analytic conductor $\Delta_{K} \cdot (1+|t|)^{\deg K}$. 
The corresponding ``convexity bound'' would correspond to the statement with $\eta_{0},\eta_{2}=- \varepsilon$ and $\eta_{1} = \tfrac{1}{4}+ \varepsilon$ for some arbitrarily small $\varepsilon > 0$
 \cite{I97}.

\begin{remark}
 It would also be pertinent to stress that weak subconvexity may not suffice when $\deg K \xrightarrow[]{} \infty$. 
 This is because we need the constant $C^{\deg K}$ in front of~\cref{eq:lind}
 to be absorbed in $\Delta_{K}^{-\eta_{0}}$ to make our error term disappear (c.f.~\cref{hy:disc}). 
 It is expected that a constant $C^{\deg K}$ would be present with any subconvexity bound of the style in~\cref{eq:lind}. 
\end{remark}


\item 
\subsubsection*{Lower bound on the discriminant}

One also needs the following hypothesis for dealing with the exponential factors $C(\eta_{1},\eta_{2})$ 
that appear in~\cref{hy:lind}.
\begin{hypothesis}
\label{hy:disc}
\begin{equation}
    \text{ For $K \in \mathcal{S}$, }\frac{\deg K}{\log\Delta_{K} } \rightarrow  0 \text{ as } \Delta_K \rightarrow  \infty.
\end{equation}
\end{hypothesis}
It is known that counterexamples to~\cref{hy:disc} exist. This is due to Golod and Shafarevich's work on infinite class field towers of ramified extensions constructed using class field theory ~\cite{GS1964}. 

Odlyzko proved that $\deg K / \log \Delta_{K}$ is universally bounded \cite{O1990}.
A bound $\deg K / \log \Delta_{K} < c$ may also suffice if $c>0$ is small enough relative to $\eta_0$ and $\log C(\eta_0,\eta_1,\eta_2)$ from \cref{eq:lind}; see \cref{eq:using_hy_disc_final}.
It is not clear what the interest in such a scenario may be.

\item 
\subsubsection*{Lower bound on residues of Dedekind zeta function}

We will also need to assume the following hypothesis.
\begin{hypothesis}
	One has for $K \in \mathcal{S}$ 
\begin{equation}
 \frac{\log  |\Res_{s=1} \zeta_K(s) |}{ \log \Delta_{K}}  \rightarrow  0 \text{ as } \Delta_{K} \xrightarrow[]{} \infty.
\end{equation}
\label{hy:residue_estimate}
\end{hypothesis}

This hypothesis is actually implied by~\cref{hy:disc} if one assumes that all the number fields $K \in \mathcal{S}$ are Galois, or if one assumes the Generalized Riemann hypothesis (GRH). 
This implication is due to the Brauer-Siegel theorem (c.f.~\cite{S1974}).

In fact, for the case of normal number fields, Stark \cite{S1974} proves a much stronger result. We will discuss this in~\cref{se:cyclo_satisfy}.
\end{enumerate}

One can then state the theorem~\cref{th:main} in the following generality.
\begin{theorem}
\label{th:general}
Let $\mathcal{S}$ be set of number fields satisfying~\cref{hy:lind},~\ref{hy:disc} and~\ref{hy:residue_estimate}. Assume also that for $K \in \mathcal{S}$ we have $r_{1} \eta_{1} + 2r_{2} \eta_{2} > 2$, where $\eta_{1},\eta_{2}$ are as in~\cref{hy:lind}.
Let $\eta_{0}$ be as in~\cref{eq:lind}.
For each $K \in \mathcal{S}$, choose $B \subseteq K \otimes \mathbb{R}$ to be an origin-centered ball of volume $V$ with respect to the trace form (\ref{eq:norm}). Then one has
\begin{equation}
\mathbb{E}(\card B \cap \Lambda) = 1 + V + \varepsilon(V,K)
\text{ for a Haar-random unit covolume }\Lambda \in \Ar(K),
\end{equation}
where as $\Delta_K \rightarrow \infty$, 
we have for any $0 < \eta <  \eta_{0}$ 
some constant $c = c(\eta, \mathcal{S})$ such that
\begin{equation}
|\varepsilon(V,K)| \leq c \cdot  \sqrt{V}  \Delta_K^{{-\eta}} .
\end{equation}

\end{theorem}

The proof of~\cref{th:main} uses the following hypotheses for the case of cyclotomic number fields, that is when
\begin{equation}
    \mathcal{S}=  \mathcal{S}_{\mathrm{cyc}} = \{ \mathbb{Q}(\mu_{n}) \mid \mu_{n} \in \overline{\mathbb{Q}} \text{ is a primitive } n\text{th root of unity}\}.
\end{equation}

We will further elaborate on this in~\cref{se:cyclo_satisfy} on why the above hypotheses hold true for the case of cyclotomic fields.

\subsection*{Related work and proof techniques}

In terms of technique, this work draws inspiration from \cite[Proposition 12.5]{ELMV11} which is related to proving Duke's theorem for cubic fields. In their paper, Einsiedler-Lindenstrauss-Michel-Venkatesh showed {a version of} Theorem \ref{th:main} when $K$ varies over cubic fields of growing discriminant by using Hecke integration formula and subconvexity, just like we use it in this paper. 
Combined with dynamical ideas, for $\SL_3(\mathbb{R})/\SL_3(\mathbb{Z})$ this was sufficient to prove equidistribution of measures supported on these ideal lattices as $\Delta_K \rightarrow \infty$. 
Since in our case the degree of the number field $K$ may go to infinity, it is no longer meaningful to talk about equidistribution. 
In the situation with growing degree, Theorem \ref{th:main} is a variant of Duke's theorem for cyclotomic fields.

Here is a rough outline of the proof of \cref{th:main}. The main input is the Hecke integration formula given in \cref{th:hecke_original} which lets us average Epstein zeta functions on ideal lattices. One can then use this formula to count the average number of lattice points using an inverse Mellin transform. 
This is done in \cref{co:contour_shifting} using a contour shifting argument which allows writing the average as a residue term and a certain contour integral involving the Dedekind zeta function of the number field $K$. The former is our main term and the latter is the error term. 
There is a small obstruction in contour shifting, which is that the decay rate of the integrand may not be sufficient to allow moving the vertical contour line. 
So in \cref{ss:hif_ind} we show that when $\deg K$ is large enough, ratios of gamma functions decay fast enough to allow contour shifting. 
Finally, one used the hypotheses stated earlier to hammer the nail,
which allows
one to show that the error term in \cref{th:main} vanishes as $\deg K \rightarrow \infty$.

We prove that the Hypotheses \ref{hy:lind}, \ref{hy:disc} and \ref{hy:residue_estimate} 
are satisfied for cyclotomic number fields using two deep theorems, 
namely subconvexity bounds by Petrow-Young \cite{PY23} and the other being an effective version of the Brauer-Siegel theorem due to Stark \cite{S1974}.

\subsection*{Organization of the paper}

This paper has been written for an interdisciplinary audience and we have tried to make it self-contained. In particular, we spend a considerable effort in reintroducing the classical work of Hecke which is well-known in analytic number theory and geometry of numbers.

Section \ref{se:setup} contains a recollection of classical facts. 
A reader familiar with these concepts can freely skip this section.
Section \ref{se:hecke} contains a description of the Hecke integration formula, written in a language suitable for our application, along with estimates on gamma factors. 
Finally, Section \ref{se:cyclo_satisfy} contains a discussion on subconvexity, Stark's theorem and the proof of~\cref{hy:lind},~\cref{hy:disc} and~\cref{hy:residue_estimate} for the case of cyclotomic fields.

\section*{Acknowledgements}
While writing this paper, we have enjoyed talking to Vlad Serban, Jialun Li, Henry Bambury, Phong Nguyen and Seungki Kim. 
We thank Petru Constantinescu for pointing us to some relevant subconvexity literature. We thank Philippe Michel, Wei Zhang and ``GH from MO''\footnote{This is a MathOverflow handle whose human identity we do not know. See \cite{MOconv}.
} 
for pointing to the state-of-the-art knowledge about subconvexity for Dirichlet L-functions.

Nihar Gargava has received funding from the ERC Grant\footnote{
Funded by the European Union. Views and opinions expressed are however those of the author(s) only and do not necessarily reflect those of the European Union or the European Research Council Executive Agency. Neither the European Union nor the granting authority can be held responsible for them.
}
101096550 titled ``Integrating Spectral and Geometric data on Moduli Space'' 
and from the SNSF grant 200021 184927 titled ``Optimal Configurations in Multidimensional Spaces'' while working on this project. 
Maryna Viazovska's research was supported by the SNSF grants  200021 184927 ``Optimal configurations in multidimensional spaces''  and  200020 215337 ``Sphere packing, Energy minimization and Fourier interpolation''.

This paper was written originally for the pre-AI era of mathematics. All ideas are from natural human intelligence.
However, since then we have managed to do an AI-assisted formal verification of~\cref{th:general} in Lean.
The results are publicly uploaded to the Palomar Registry~\cite{GargavaMeanValueRepo}.
This has led to the correction of some minor mathematical details in the final version.

\section{Preliminaries}
  \label{se:setup}

Let $K$ be a number field of degree $d=r_1 + 2r_2$ where $r_1$ is the number of real embeddings and $r_2$ is the number of pairs of complex embeddings, paired under the complex conjugation. 
We denote by $x \mapsto \overline{x}$ the complex conjugation on each of the $r_2$ pairs of complex embeddings $K \rightarrow \mathbb{C}$.
Let $\Delta_K$ be the absolute value of the discriminant of $K$.
Let $\KR = K \otimes_{\mathbb{Q}} \mathbb{R}$.
We endow $\KR$ with the positive-definite quadratic form
\begin{equation}
\label{eq:norm}
  x \mapsto \Tr(x \overline{x}).
\end{equation}

 We denote by $\OK$ the ring of integers of $K$.  In the geometry of numbers, it is a classical fact that $\Delta_K$ is the square of the covolume of $\OK \subseteq \KR$ with respect to this quadratic form. That is
\begin{equation}
	\sqrt{\Delta_{K} } = \vol( \KR/\OK ).
\end{equation}

We see every ideal $\mathcal{I} \subseteq \OK$ as a lattice in $\KR$. With respect to this, the ideal $\mathcal{I}$ has 
\begin{equation}
\vol(\KR/\mathcal{I}) = \covol(\mathcal{I}) = \N(\mathcal{I}) \sqrt{\Delta_K},
\label{eq:vol}
\end{equation}
where $\N(\mathcal{I}) = \card \OK/\mathcal{I}$, which is known to be finite. 
See~\cite{Neu13} for more background information.

\begin{remark}
With a different choice of the norm in Equation (\ref{eq:norm}), it is common to have a factor of $2^{-r_2}$ in Equation (\ref{eq:vol}). See \cite[Chapter I-\S 5]{Neu13}.
\end{remark}

\begin{remark}
\label{re:peculiar_norm}
To get an idea of the quadratic form in~\cref{eq:norm}, observe that 
the region $[-1,1]^{r_1} \times \{x \in \mathbb{C} \mid |x| \leq 1\}^{r_2} \subseteq \mathbb{R}^{r_1} \times \mathbb{C}^{r_2}\simeq \KR$  should have a volume of $2^{r_1} (2 \pi)^{r_2}$ (and not $2^{r_1} \pi^{r_2}$).

For example, with $K =\mathbb{Q}( \sqrt{-1})$, $1 \in K \subseteq K_{\mathbb{R}}$ is at a distance $\sqrt{2}$ from the origin.
\end{remark}

\subsection{Arakelov class group}
\label{se:arakelov}

Let $\Ar(K)$ be the Arakelov class group of $K$, also known as the extended class group in \cite{MV07}. 
The Arakelov class group $\Ar(K)$ 
is defined as the set of ideal lattices $(g, \mathcal{I})$ where $\mathcal{I}\subseteq K$ is a non-zero finitely generated $\OK$-module
and $g \in K_\mathbb{R}^{\times}$. We identify two ideal lattices as per the rule
\begin{equation}
	(g_1, \mathcal{I}_1) = (g_2,\mathcal{I}_2) \Leftrightarrow g_1 \mathcal{I}_1  = c\cdot g_2 \mathcal{I}_2 \text{ for some }c>0.
\end{equation}
Here the equality on the right is the equality of sublattices in $\KR$. 

Observe that the following sequence of topological abelian groups is exact.
\begin{equation}
     \{1\} \rightarrow \KR^{(1)}/\OK^{\times} \rightarrow \Ar(K) \rightarrow \Cl(K)  \rightarrow \{1\}
.\end{equation}

Here $K_\mathbb{R}^{(1)}$ is the group 
\begin{equation}
  \KR^{(1)}= \{a \in \KR \mid |\N(a)| = 1\}.
\end{equation}
The map $\N:K_\mathbb{R} \rightarrow \mathbb{R}$ is the norm map. Dirichlet's unit theorem tells us that $\KR^{(1)}/\OK^{\times}$ is compact.

One can renormalize every ideal lattice to be of unit determinant by {identifying it with the following sublattice} of $\KR$:
\begin{equation}
	(g, \mathcal{I}) \mapsto \left( \sqrt{\Delta_K} \cdot |\N(g)| \cdot \N(\mathcal{I}) \right) ^{-\frac{1}{d}} \cdot g \mathcal{I}.
\end{equation}
Hence, from now on, when we say $\Lambda \in \Ar(K)$ is an ideal lattice, we mean a unit covolume lattice.

\subsubsection{Haar measure on the Arakelov Class group}
\label{se:arakelov_haar}
Since $\Ar(K)$ is a locally compact topological group, it admits a Haar measure. If we fix a Haar measure $ \diffx a$ on $K_\mathbb{R}^{(1)}$, 
we get a natural measure on $\Ar(K)$ as $\card \Cl(K)$ number of copies of $\KR^{(1)}/\OK^{\times}$. The most natural choice of a Haar measure on $\KR^{(1)}$ is 
\begin{equation}
	\int_A \diffx a =  \int_{ A \cdot (0,1] \subseteq \KR} \diff x  , \text{ for any measurable } A \subseteq \KR^{(1)},
  \label{eq:haar_measure_on_KR1}
\end{equation}
where $\diff x$ is the Lebesgue measure on $\KR$ with respect to the inner product given in Equation (\ref{eq:norm}).
Throughout the paper, we will assume $\KR^{(1)}$ has this measure. 

\begin{remark}
\label{re:connection_between_measures}
To see how the measure $\diff x$ on $\KR$ and $\diffx a $ on $\KR^{(1)}$ are related, we observe that for $A \subseteq \KR^{(1)}$ and $T > 0$, we have for $\varepsilon \rightarrow 0$
\begin{equation}
	\int_{A \cdot (T, T+ \varepsilon )} \diff x =  \varepsilon \, (dT^{d-1})\,  \int_{A} \diffx a + O(\varepsilon^{2}).
\end{equation}
\end{remark}

\subsection{Dedekind zeta function}

Recall that the Dedekind zeta function of a number field is the function
\begin{equation}
\zeta_K(s) = 
\sum_{\substack{\mathcal{I} \subseteq \OK \\ \mathcal{I} \text{ ideal}}} 
\frac{1}{\N(\mathcal{I})^{s}} \text{ for }\Re(s) > 1.
\end{equation}
This function can get meromorphically continued to the complex plane with a simple pole at {$s=1$}. The residue at {$s=1$} is the one that will concern us the most.

\subsubsection{Residue at 1}
Using a standard argument, one observes readily that  
\begin{equation}
   \Res_{s=1} \zeta_K(s)
	 = 
	  \lim_{T \rightarrow \infty}\frac{ \card \{\text{Ideals }\mathcal{I} \subseteq \OK \mid \N(\mathcal{I}) \leq T \}  }{T}.
\end{equation}

By a classical 
geometry of numbers argument, one can precisely evaluate this residue. First we need the following lemma which will also be of use eventually.
\begin{lemma}
	\label{le:counting_ideals}
	Let $\mathcal{I} \subseteq K$ be a fractional ideal.
	There is a bijection between $w \in (\mathcal{I}\setminus \{0\}) / \OK^{\times}$ and the ideals in the ideal class $[\mathcal{I}^{-1}] \in \Cl(K)$ given by 
	\begin{equation}
	  w \mapsto  w \mathcal{I}^{-1}.
	\end{equation}
\end{lemma}
\begin{proof}
	If $w \in \mathcal{I}$ then the $\OK$-module $\langle w \rangle$ generated by $w$ lies in $\mathcal{I}$. 
	Hence, there exists a fractional ideal $\mathcal{J} \subseteq K$ such that $ \langle w \rangle  = \mathcal{J}\mathcal{I}$. 
	One can then check that $\mathcal{J} \subseteq \OK$ and $\mathcal{J}$ is actually an ideal. Clearly $[\mathcal{J}] = [\mathcal{I}^{-1}]$

	On the other hand, suppose we are given an ideal $\mathcal{J}\subseteq \OK$ in the ideal class of $[\mathcal{I}^{-1}]$. Then, we must have that $\mathcal{J}\mathcal{I} = \langle w \rangle$ for some $w \in K^{\times}$. Since $\langle w \rangle \subseteq  \mathcal{I}$, this uniquely identifies $w \in (\mathcal{I} \setminus \{0\})/ \OK^{\times}$.
\end{proof}

We are now ready to prove the following {classical result, known as analytic class number formula}.
\begin{lemma}
\label{le:residue_of_dz}
\begin{equation}
\Res_{s=1} \zeta_K(s) = \frac{\vol( \Ar(K)) }{\sqrt{ \Delta_K} }.
\end{equation}
\end{lemma}
\begin{proof}
We will only sketch the proof. 
Let us fix an ideal class $[\mathcal{I}^{-1}]\in \Cl(K)$. Then, we get from 
Lemma \ref{le:counting_ideals} that the set of ideals $\mathcal{J} \subseteq \OK$ in this ideal class are in bijection with 
the set $( \mathcal{I}\setminus \{0\} )/\OK^{\times}$. So we get that 
\begin{equation}
	\card \{ \text{ Ideals } \mathcal{J} \in [\mathcal{I}^{-1}] \mid \N(\mathcal{J}) \leq T  \} = 
	\card \{ w\in (\mathcal{I}\setminus \{0\})/\OK^{\times} \mid \N(w) \leq  T \N(\mathcal{I}) \}
\end{equation}
This is the same thing as counting lattice points in a domain in $\KR$. That is, the count equals
\begin{equation}
  \card \left( \{ x \in \KR^{\times} \mid \N(x) \leq T \N(\mathcal{I})\} \cap \mathcal{I}\right)/ \OK^{\times} .
\end{equation}
One can then observe that 
\begin{equation}
	\vol(\KR/\mathcal{I}) = \N(\mathcal{I}) \sqrt{\Delta_K}.
\end{equation}
And so the point count is asymptotically (for large $T$)
\begin{equation}
	\card \{ \text{ Ideals } \mathcal{J} \in [\mathcal{I}^{-1}] \mid \N(\mathcal{J}) \leq T  \} \sim  \frac{\vol(\KR^{(1)}/\OK^{\times})}{\sqrt{\Delta_K} \N(\mathcal{I})}  \cdot (T \N(\mathcal{I})),
\end{equation}
where the volume on $\KR^{(1)}$ is taken with respect to the Haar measure defined in Section \ref{se:arakelov_haar}. Doing this for each ideal class and summing up gives the required statement.
\end{proof}

\subsection{Mellin transform}

Let $g : \mathbb{R}_{>0} \rightarrow \mathbb{C}$ be a compactly supported smooth function (so the support is bounded away from $0$). We define the Mellin transform of $g$ to be 
\begin{equation}
	\mell{g} (s) =  \int_{0}^{\infty} g(t)t^{s-1} \diff t.
\end{equation}
This is a holomorphic function on all of the complex plane. One can recover the function $g$ using the inverse Mellin transform as 
\begin{equation}
	g(t) = \frac{1}{2 \pi i } \int_{\sigma-i \infty}^{\sigma + i \infty} t^{-s} \mell{g}(s) \diff s ,
\end{equation}
where we are free to choose any $\sigma \in \mathbb{R}$. 

However, our main interest is in using the theory of the Mellin transform for the function $\ind_{[0,R]}$, which is the indicator function of the interval $[0,R] \subseteq \mathbb{R}$ for some $R > 0$. For this situation, observe that 
\begin{equation}
\mell{\ind_{[0,R]}}(s) = \tfrac{1}{s}R^{s} {\text{ for } s \neq 0, }
\label{eq:inverse_mellin}
\end{equation}
which is a holomorphic function on $\Re(s)>0$. 
In this case, one can still nearly recover the function {$\ind_{[0,R]}$ on $\mathbb{R}_{>0}$ }using an improper integral.
\begin{lemma}
	Suppose that $t\in \mathbb{R}_{>0}$  and $t \neq R$. Then
\label{le:improper}
\begin{equation}
	\ind_{[0,R]}(t) = \lim_{T \rightarrow \infty} \frac{1}{2\pi i}  \int_{\sigma - iT}^{\sigma +iT} t^{-s}\frac{R^{s}}{s} \diff{s} \text{ for any $\sigma>0$}. 
\end{equation}
For $t=R$, the right hand side is equal to $1/2$.
\begin{proof}
	Follows from residue calculations.
	See \cite[(5.3),(5.5)]{MV06} for instance.
\end{proof}
\end{lemma}

Therefore, to avoid worrying too much about $t=R$, from henceforth we will assume that $\ind_{[0,R]}(R) = \tfrac{1}{2}$.

\section{Hecke integration formula}
\label{se:hecke}
Let $f_{\rad} : \mathbb{R}_{>0} \rightarrow \mathbb{C}$ be a compactly supported smooth function.
One can then use this to define a compactly supported radial function  
$f: K_\mathbb{R} \rightarrow \mathbb{R}$ by taking
\begin{equation}
  f(x)  = f_{\rad}(\|x\|),
\end{equation}
where $\|x\|$ is as defined in Equation (\ref{eq:norm}). 
We can then recover the function $f(x)$ by Equation (\ref{eq:inverse_mellin}) as
\begin{equation}
  f(x) = 
  \frac{1}{2\pi i} \int_{\sigma - i \infty}^{\sigma + i \infty} \|x\|^{-s} \mell{f_{\rad}}(s) \diff s \text{ for any } \sigma > 0.
\end{equation}

Given an ideal lattice $\Lambda \subseteq K_\mathbb{R}$, 
one can consider the sum given by 
\begin{equation}
\sum_{v \in \Lambda \setminus \{0\}} f(v)    
 = 
{\sum_{v \in \Lambda \setminus \{0\}} 
  \frac{1}{2\pi i} \int_{\sigma - i \infty}^{\sigma + i \infty} \|v\|^{-s} \mell{f_{\rad}}(s) \diff s .}
  \label{eq:inverse_mellin_radial}
\end{equation}
We would like to pull out the integral outside the sum. However, to have that 
{$\sum_{v \in \Lambda \setminus \{0\}} \|v\|^{-s}$ } is absolutely convergent, we assume that $\sigma > d$ due to the following lemma.

\begin{lemma}
\label{le:abs_conv_epstein}
Given a lattice $\Lambda \subseteq \mathbb{R}^{d}$ and a suitable inner product on $\mathbb{R}^{d}$, the {following sum is absolutely convergent and defines a holomorphic function on the half-plane $\Re(s)> d$:
\begin{equation}
  \Eps(\Lambda,s) = \sum_{v \in \Lambda \setminus \{0\}} \|v\|^{-s}.
\end{equation}}
It admits an analytic continuation to $\mathbb{C} \setminus \{d\}$ with a pole at $s=d$ and residue:
\begin{equation}
	\Res_{s=d} \Eps(\Lambda,s) =d \cdot\Res_{s=1} \Eps(\Lambda,sd) = d  \cdot\frac{  \pi^{\frac{d}{2}}}{\Gamma(\tfrac{d}{2} + 1)} \cdot \frac{1}{\covol(\Lambda)}.
\end{equation}
The completed Epstein zeta function satisfies the following functional equation when $\covol(\Lambda)=1$.
\begin{align}
	\Eps^{*}(\Lambda,s) &=  {\pi^{-\frac{s}{2}}} { \Gamma(\tfrac{s}{2}) }\Eps(\Lambda,s),\\
	\Eps^{*}(\Lambda,s)  & = \Eps^{*}(\Lambda^{*},d-s).
	\label{eq:completed_epstein}
\end{align}
Here $\Lambda^{*} = \{y \in \mathbb{R}^d \mid \langle y , \Lambda \rangle \subseteq \mathbb{Z}\}$ is the dual lattice of $\Lambda$.
\begin{proof}
Absolute convergence and the value of the residue can be shown by counting lattice points $x \in \Lambda $ such that $\|x\| \leq R$ as $R \to \infty$. This is roughly the volume of a ball of radius $R$ normalized by the size of a fundamental domain of $\mathbb{R}^{d}/\Lambda$.

Analytic continuation and the functional equation follows from the Poisson summation formula. 
One can see all details worked out in Neukirch \cite{Neu13}.
\end{proof}
\end{lemma}

The function $\Eps(\Lambda,s)$ is called the Epstein zeta function of the lattice $\Lambda$. 
One can then get the following lemma.
\begin{lemma}
\label{le:contour_integral}
Let $\Lambda \subseteq \mathbb{R}^{d}$ be a unit covolume lattice and $f(x) = f_{\rad}(\|x\|)$ be a radial function on $\mathbb{R}^{d}$ for a compactly supported smooth function $f_{\rad}:\mathbb{R}_{>0}\rightarrow \mathbb{R}$. Then
\begin{equation}
	\sum_{v \in \Lambda } f(v) 
	= \frac{d}{2\pi i} \int_{\sigma - i\infty}^{\sigma + i \infty}  
	\Eps(\Lambda,sd) \mell{f_{\rad}}(sd) \diff s \text{ for $\sigma  > 1$}.
\end{equation}
\end{lemma}
\begin{proof}
	Follows from absolute convergence of $\Eps(\Lambda,s)$ for $\Re(s)>d$ and Equation (\ref{eq:inverse_mellin_radial}). Note the change of variable from $s$ to $sd$.
\end{proof}

Before moving to the main theorem in this section, let us reestablish the so-called
``Hecke's trick'' \cite[Lemma 3.5]{K20}. 
\begin{lemma}
    \label{le:hecke_trick}
{For $\Re(s)>1$} we have that 
\begin{align}
\int_{a \in \KR^{(1)}} \|av\|^{-sd} \diffx a
& =  |\N(v)|^{-s} \cdot 
{  \frac{(2 \pi )^{r_2} \Gamma(\tfrac{s}{2})^{r_1} \Gamma(s)^{r_2} }{ 2^{sr_2} \cdot \frac{d}{2} \Gamma(\frac{sd}{2})  } }.
\end{align}
\end{lemma}
\begin{proof}
Let us prove it for $v=1$. 
Consider the integral
\begin{equation}
\label{eq:first_one}
  \int_{x \in \KR} e^{-\|x\|^{2}} |\N(x)|^{s-1} \diff x
  = \int_{x \in \KR^{\times}} e^{-\|x\|^{2}} |\N(x)|^{s}  |\N(x)|^{-1}\diff x.
\end{equation}
We can evaluate it in two ways. 
One way is to write the Haar measure on $|\N(x)|^{-1} \diff x$ on $\KR^{\times}$ by writing (see Remark \ref{re:connection_between_measures})
\begin{align}
	\KR^{\times}  & \simeq \mathbb{R}_{>0} \times \KR^{(1)} \\
	x  & = t \cdot a \\
	|\N(x)|^{-1} \diff x & =  \deg K \cdot \left( \tfrac{1}{t}\diff t  \right)\cdot \diffx a.
\end{align} 
Observe then that for $d=\deg K$ we have $\N(at) =t^{d}\N(a)$. Hence,
\begin{align}
  \int_{x \in \KR} e^{-\|x\|^{2}} |\N(x)|^{s-1} \diff x
  & = 
  \int_{t \in \mathbb{R}_{>0}}\int_{a \in \KR^{(1)}} e^{-t^2\|a\|^{2}} d t^{sd-1} \diffx a \diff t .\\
  & = 
  \int_{t \in \mathbb{R}_{>0}}e^{-t^2} d t^{sd-1} \diff t  \int_{a \in \KR^{(1)}} \|a\|^{-sd} \diffx a  \\
  & = 
  \tfrac{d}{2}\Gamma(\tfrac{sd}{2})
  \int_{a \in \KR^{(1)}} \|a\|^{-sd}\diffx a.
\end{align}
On the other hand, one can also write $x=(x_1,\dots,x_{r_1+r_2}) \in \mathbb{R}^{r_1} \times \mathbb{C}^{r_2}$. 
\begin{align}
  \int_{x \in \KR} &  e^{-\|x\|^{2}} |\N(x)|^{s-1} \diff x \\
  & = 
  \int_{x \in \KR} e^{-\left( \sum_{i=1}^{r_1}|x_i|^{2} +2 \sum_{i=r_1+1}^{r_2}|x_i|^{2} \right)} |x_1 x_2 \dots x_{r_1}|^{s-1} |x_{r_1+1} \dots x_{r_1+r_2}|^{2( s-1 )} \diff x\\
  & = 
\prod_{i=1}^{r_1}\left( 2\int_{x \in {\mathbb{R}_{> 0}} } e^{- x^{2}} x^{s-1} \diff x \right)
\prod_{i=r_1+1}^{r_1+r_2}\left( 2\cdot 2 \pi \int_{x \in {\mathbb{R}_{> 0}} } e^{- 2 x^{2}} x^{2(s-1)} x\diff x \right)\\
  & =  (2\pi)^{r_2} 2^{-sr_2} \Gamma(\tfrac{s}{2})^{r_1} \Gamma(s)^{r_2}.
\end{align}
The factor of $2$ appearing before $2\pi$ is a bit subtle. See Remark \ref{re:peculiar_norm}.

For $v \neq 1$, replace $x$ with $xv$ in the integral in Equation (\ref{eq:first_one}) and follow the same process.

\end{proof}

Suppose now if we equip $\Ar(K)$ with a Haar-probability measure, then one can ask what is the value of 
\begin{equation}
  \int_{\Lambda \in \Ar(K)}\left( \sum_{v \in \Lambda \setminus \{0\}} f(v)\right) \diff \Lambda .
\end{equation}

The answer can now be given by the classical theorem of Hecke. 
Hecke originally evaluated this integral for the case of $f(v) = \|v\|_2^{-s}$ to prove the meromorphic continuation of the Dedekind zeta function \cite{H1918}.
\begin{theorem}
\label{th:hecke_original}
{\bf (Hecke, 1918)}

Suppose $\Ar(K)$ is given the Haar-probability measure $\diff \Lambda$
and suppose that each $\Lambda \in \Ar(K)$ is normalized as a unit covolume lattice in $\KR$. 
Then, for $\Re(s)>1$ we have that 
\begin{equation} 
  \int_{\Lambda \in \Ar(K)}
  \Eps(\Lambda, ds) \diff \Lambda
  =
  \frac{
  \zeta_K({s})
}{   \vol(\Ar(K))}  
\cdot 
\frac{\pi^{\frac{sd}{2}}}
{ \Gamma(\frac{sd}{2})\cdot \frac{d}{2}} 
\cdot
\frac{ \Delta_K^{\frac{s}{2} } \left( \pi^{-\frac{s}{2}} \Gamma(\tfrac{s}{2})  \right)^{r_1}  \left(   (2\pi)^{-s}\Gamma(s)\right)^{r_2}   }
{(2\pi)^{-r_2}}
  .
\end{equation}
\begin{remark}
\label{re:sanity}
One way to sanity check that the above expression is correct is to know that the completed Dedekind zeta function
\begin{equation}
  \zeta^{*}_{K}(s) = 
{ \Delta_K^{\frac{s}{2}} \left( \pi^{-\frac{s}{2}} \Gamma(\tfrac{s}{2})  \right)^{r_1}  \left(   (2\pi)^{-s}\Gamma(s)\right)^{r_2}   } \zeta_{K}(s)
\end{equation}
must satisfy the functional equation $\zeta_{K}^{*}(1-s) = \zeta_K^{*}(s)$. Then, using (\ref{eq:completed_epstein}), we know that the function $\zeta_{K}^{*}(s)^{-1} \int_{\Lambda \in \Ar(K)}  \Eps^{*}(\Lambda,d s) \diff \Lambda$ 
must be free of Gamma factors since $\Ar(K)$ is symmetric with respect to $\Lambda \mapsto \Lambda^{*}$. Hence, it is likely to be a constant and we can then chase this constant by taking $s \rightarrow 1$.
\end{remark}

\end{theorem}

\begin{proof}
{\bf (of \cref{th:hecke_original})}

Even though we had the Haar-random measure on $\Ar(K)$, for the proof
we will work with the natural measure on $\Ar(K)$ as described in Section \ref{se:arakelov_haar}. Also, we will assume that $\Re(s) > 1$ since this assumption is sufficient to show the equality.
Let $V_{C} = \vol(\Ar(K))$ for notational clarity. 

Denote by $\diffx a$ the measure on $\KR^{(1)}$ given in Equation (\ref{eq:haar_measure_on_KR1}).
Let $\Lambda \in \Ar(K)$ be an ideal lattice. 
Then $\KR^{(1)} \cdot \Lambda$ 
is one complete fiber of the map $\Ar(K) \rightarrow  \Cl(K)$ containing $\Lambda$. 
The stabilizer of the $\KR^{(1)}$-action is exactly 
$\OK^{\times} \subseteq \KR^{(1)}$. 
So, if $h = \card \Cl(K)$ and if $\Lambda_1,\dots,\Lambda_h$ is a set of ideal lattices, one from each ideal class, then after some simple folding-unfolding
\begin{align}
  \int_{\Lambda \in \Ar(K)} \Eps(\Lambda,sd)
  \diff \Lambda
  & =  \frac{1}{V_{C}} 
  \sum_{i=1}^{h} \int_{ a\in \KR^{(1)} /\OK^{\times}}
  \left( \sum_{v \in \Lambda_{i} \setminus \{0\}} \|av\|^{-sd}  \right) \diffx a\\
  & =  \frac{1}{V_{C}}
  \sum_{i=1}^{h} \int_{ a\in \KR^{(1)} / \OK^{\times} }\left( \sum_{u \in \OK^{\times}} \sum_{v \in ( \Lambda_{i} \setminus \{0\} )/\OK^{\times}}  \| a u v \|^{-sd}\right) \diffx a\\
  & =  \frac{1}{V_C} \sum_{i=1}^{h} \int_{ a\in \KR^{(1)} }\left(  \sum_{v \in ( \Lambda_{i} \setminus \{0\} )/\OK^{\times}}  \| a  v \|^{-sd} \right) \diffx a.
\end{align}

Then, we can bring in 
Hecke's trick from Lemma \ref{le:hecke_trick}
and get that if $v \in \KR^{\times}$ then
\begin{align}
\int_{a \in \KR^{(1)}} \|av\|^{-sd} \diffx a
& =  |\N(v)|^{-s} \cdot 
{  \frac{ (2\pi)^{r_2} \Gamma(\tfrac{s}{2})^{r_1} \Gamma(s)^{r_2} }{  2^{sr_2} \cdot \frac{d}{2} \Gamma(\frac{sd}{2})  } }.
\end{align}

Now without loss of generality, one can assume that $\Lambda_i \subseteq K$ is a fractional ideal. 
This means that $v \in \Lambda_{i} \setminus \{0\} \Rightarrow |\N(v)|\neq 0$. Note that $|\N(v)|= |\N(uv)|$ for any $u \in \OK^{\times}$ so it is well defined as a function on $( \Lambda_i \setminus \{0\} ) / \OK^{\times}$. Hence, we can write that 
\begin{equation}
  \int_{\Lambda \in \Ar(K)}  \Eps(\Lambda,sd)\diff \Lambda = 
    \frac{1}{V_C} \sum_{i=1}^{h} \left(  \sum_{v \in ( \Lambda_{i} \setminus \{0\} )/\OK^{\times}}  |\N(v)|^{-s} \right) \cdot
{  \frac{ (2\pi)^{r_2} \Gamma(\tfrac{s}{2})^{r_1} \Gamma(s)^{r_2} }{  2^{sr_2}\frac{d}{2} \Gamma(\frac{sd}{2})  } }.
\end{equation}

Observe that 
\begin{equation}
\sum_{i=1}^{h} \sum_{v \in ( \Lambda_{i} \setminus \{0\} )/\OK^{\times}}|\N(v)|^{-s}  = \Delta_K^{\frac{s}{2}} \zeta_K({s}).
\end{equation}
Indeed, we have that for some ideal $\mathcal{I}_{i} \subseteq \OK$, 
\begin{equation}
  \Lambda_i =  \Delta_K^{-\frac{1}{2d}} \N(\mathcal{I}_i)^{-\frac{1}{d}} \mathcal{I}_i.
\end{equation}
So, for some $v \in \Lambda_i$, we have some $w \in \mathcal{I}_i$ such that 
$$v = \Delta_K^{-\frac{1}{2d}} \N(\mathcal{I}_i)^{-\frac{1}{d}}w \Rightarrow |\N(v)| =  |\N(w)| \cdot \Delta_K^{-\frac{1}{2}}  \cdot \N(\mathcal{I}_i)^{-1}.$$
Then, observe that as $w$ varies through the different representatives of $\mathcal{I}_i/\OK^{\times}$, we must have that $\N(w)/ \N(\mathcal{I}_i)$ goes through norms of different ideals lying in the ideal class of $\mathcal{I}_{i}^{-1}$.
This was shown in Lemma \ref{le:counting_ideals}.

\end{proof}

After this, we would like to use Lemma \ref{le:contour_integral} to average $\sum_{v \in \Lambda \setminus \{0\}}^{} f(v)$. This gives a completely explicit formula.
  \begin{corollary}
	  \label{co:contour_shifting}
  Consider the same setting as Theorem \ref{th:hecke_original}. Let $f_{\rad}:\mathbb{R}_{>0}\rightarrow \mathbb{R}$ be a compactly supported smooth function and $f:\KR \rightarrow \mathbb{R}$ be $f(x)=f_{\rad}(\|x\|)$ as before. Then
  \begin{equation}
	  \int_{\Lambda \in \Ar(K)} \left( \sum_{v \in \Lambda\setminus\{0\}}f(v) \right) \diff \Lambda =  
\int_{x \in \KR} f(x) { \diff x }+ \varepsilon(K,f),
  \end{equation}
  where 
\begin{equation}
	\varepsilon(K,f) = \frac{  (2\pi)^{r_2}}{ \vol(\Ar(K)) \pi i } \int_{\frac{1}{2} - i \infty}^{\frac{1}{2} + i \infty}  \Delta_K^{\frac{s}{2}}\zeta_{K}(s)\mell{f_{\rad}}(sd) \left( \frac{\Gamma(\frac{s}{2})^{r_1} \Gamma(s)^{r_2}}{ 2^{sr_2} \Gamma(\frac{sd}{2})}\right)\diff s
  \end{equation}
  \end{corollary}
\begin{remark}
\label{re:alternate}
There is perhaps a more elegant but less explicit version of the error term $\varepsilon(K,f)$. First, define the ``Arakelov-Mellin transform'' of $f:\KR \rightarrow \mathbb{R}$ to be 
\begin{equation}
	\mellK{f}(s) =  \int_{\KR}  f(x) |\N(x)|^{s-1} \diff x.
\end{equation}
Then, it turns out that for $f$ considered in Theorem \ref{th:hecke_original} and Corollary \ref{co:contour_shifting}, we have 
\begin{equation}
\label{eq:radial_K_mellin}
	\mellK{f} (s) =  \mell{f_{\rad}}(sd) \cdot \left( 2 (2\pi)^{r_2}\frac{\Gamma(\tfrac{s}{2})^{r_1} \Gamma(s)^{r_2}}{2^{s r_2}\Gamma(\tfrac{sd}{2})}\right) .
\end{equation}

One can then write that 
\begin{equation}
	\varepsilon(K,f) = \frac{1}{2 \pi i} \int_{\frac{1}{2} - i \infty}^{\frac{1}{2} + i \infty} \Delta_K^{ \frac{s}{2}} \zeta_K (s) \mellK{f}  (s) \diff s .
\end{equation}

\end{remark}

  \begin{proof}
{\bf (of \cref{co:contour_shifting})}
  
	  Using \cref{th:hecke_original} {and \cref{le:contour_integral}}, we can write
      $ \int_{\Lambda \in \Ar(K)} \left( \sum_{v \in \Lambda\setminus\{0\}}f(v) \right) { \diff \Lambda}$
      to be the following for $\sigma > 1$:
	  \begin{equation}  
	  \frac{ (2\pi)^{r_2} }{\vol(\Ar(K)) \cdot \frac{d}{2}}
	  \frac{d}{2 \pi i} 
	  \int_{\sigma - i \infty}^{\sigma + i \infty}  \Delta_K^{\frac{s}{2}}\zeta_{K}(s)\mell{f_{\rad}}(sd) \left( \frac{\Gamma(\frac{s}{2})^{r_1} \Gamma(s)^{r_2}}{ 2^{sr_2} \Gamma(\frac{sd}{2})}\right)\diff s.
	  \end{equation}
	  Note that here we interchanged the contour integration with the integration over $\Ar(K)$ which is allowed since $\Ar(K)$ is compact. Now observe that $f_{r}$ is a smooth function and therefore the Mellin transform $\mell{f_r}$ has excellent decay along the imaginary axis whenever $\sigma$ lies in a compact interval. 
	  Since the integrand on the right hand side analytically continues to $\sigma \leq 1$, we can shift our contour to $\sigma = \frac{1}{2}$ and while doing so, we pick a residue at $s=1$. 
	  
	  Using Lemma \ref{le:residue_of_dz}, we can get the value of this residue to be 
	  \begin{equation}
		  \frac{ (2\pi)^{r_2} }{\frac{d}{2}} \frac{ d \pi^{\frac{ r_1 }{2}}}{2^{r_2} \Gamma(\frac{d}{2})} \mell{f_{\rad}}(d) 
		  =  \frac{d \pi^{\frac{d}{2}}}{\Gamma(\frac{d}{2}+1)}\int_{\mathbb{R}_{>0}} f_{\rad}(t)t^{d-1} \diff t = \int_{\KR} f(x) \diff x.
	  \end{equation}
      Here we use the fact that the surface volume of a unit sphere in $d$ dimensions is $d\pi^{\frac{d}{2}}/\Gamma(\frac{d}{2}+1)$
  \end{proof}

\subsection{Estimates for gamma factors on the critical line}

Define
\begin{equation}
G(s) =\frac{ \Gamma(\tfrac{1}{2}s)^{r_{1}} \Gamma(s)^{r_{2}}}{ \Gamma(\frac{1}{2}d s)}
\end{equation}
We will prove the following lemma.

\begin{lemma}
  \label{le:gamma_estimate}
For some constant $c_{1}>0$ which does not depend on $r_{1},r_{2}$ and $t$, we have
\begin{equation}
	\log\left| {    G(\tfrac{1}{2} + i t) }\right|  
    \leq  - \left(\tfrac{1}{4} d \log d - c_{1} d \right)  - \tfrac{r_{1}+r_{2}-1}{4}\log\left( t^{2} + \tfrac{1}{4}\right)
\end{equation}
\end{lemma}
\begin{proof}

We begin with the exact Stirling approximation formula. For $|\arg(s)|< \pi - \varepsilon$, where $\varepsilon > 0$, it states
\begin{equation}
	\label{eq:stirling}
    \log  \Gamma(s) =  \left( s- \tfrac{1}{2}\right)   \log s - s + \tfrac{1}{2}\log (2 \pi) + E(s)
\end{equation}
where 
\begin{equation}
    E(s) = \int_{0}^{\infty} \frac{ 2\arctan(\frac{x}{s})}{e^{2 \pi x} - 1} \diff x.
\end{equation}

Our main interest is in estimating for $s= 1/2+it$
\begin{equation}
    \Delta(s) =  \Re( r_{1}E(s/2) + r_{2} E(s) - E(ds/2)),
\end{equation}
since the real part alone contributes to the absolute value of $G(s)$.

We know that 
\begin{equation}
\arctan s =  \frac{1}{2i} \log\left( \frac{1 + is}{ 1 - is}\right).
\end{equation}

Therefore, when $s = \sigma + it$ for $\sigma>0$, one has 
\begin{equation}
2 \Re(\arctan (x/s)) = 
\arg \left( \frac{\sigma + i(t+x)}{\sigma + i (t-x)} \right) 
= 
\arg \left( \frac{1 + i(t+x)\sigma^{-1}}{1 + i (t-x)\sigma^{-1}} \right) 
.
\end{equation}

Denote 
\begin{equation}
    a(x,t) = \arg\left(  \frac{1 + i(t+x)}{ 1 + i(t-x)} \right).
\end{equation}
This function therefore measures the angle subtended from the origin by a segment of length $2x$ between $1+i(t+x)$ and $1-i(t-x)$.
See~\cref{fig:angle}.

\begin{figure}
    \centering

\tikzset{every picture/.style={line width=0.75pt}} 

\begin{tikzpicture}[x=0.75pt,y=0.75pt,yscale=-1,xscale=1]

\draw  (17,233.9) -- (238,233.9)(39.1,35) -- (39.1,256) (231,228.9) -- (238,233.9) -- (231,238.9) (34.1,42) -- (39.1,35) -- (44.1,42)  ;
\draw  [dash pattern={on 4.5pt off 4.5pt}]  (110,47) -- (110,257) ;
\draw    (110,47) -- (39.1,233.9) ;
\draw    (110,181) -- (39.1,233.9) ;
\draw  [dash pattern={on 4.5pt off 4.5pt}]  (55,193) .. controls (62,192) and (72,198) .. (74.55,207.45) ;

\draw (122,148) node [anchor=north west][inner sep=0.75pt]   [align=left] {};
\draw (115,171.4) node [anchor=north west][inner sep=0.75pt]    {$1+i( t-x)$};
\draw (116,241.4) node [anchor=north west][inner sep=0.75pt]    {$1$};
\draw (118,31.4) node [anchor=north west][inner sep=0.75pt]    {$1+i( t+x)$};

\end{tikzpicture}

\caption{ The function $a(x,t)$ measures for $x>0$ the angle subtended at the origin. One can argue from this picture that $a(4x,2t) \geq a(2x,2t)$ for every $t \in \mathbb{R}$ and $x>0$.}
\label{fig:angle}
\end{figure}

Then, we have an expression for $\Delta(\frac{1}{2} + it)$  as
\begin{align}
	\Delta(\tfrac{1}{2}+it) 
	& =  \int_{0}^{\infty} \frac{r_{1}a(4x,2t) + r_{2} a(2x,2t) - a(\tfrac{4}{d} x, 2t) }{ e^{2 \pi x} - 1} dx \\
\end{align}

Observe that one has $a(4x,2t) \geq a(2x,2t)$ because the angle subtended is larger. Hence, we write
\begin{equation}
    \Delta\left(\tfrac{1}{2}+it\right) \leq 
    (r_{1}+r_{2}) 
    \int_{0}^{\infty}  \frac{a(4x,2t)}{e^{2 \pi x} - 1} \diff x
\end{equation}.

For $x \geq 0$, we estimate
\begin{equation}
    a(x,t) \leq 
    \begin{cases}
	\frac{2x}{1+t^{2} - x^{2}} & \text{ when } 
\frac{2x}{1+t^{2} - x^{2}}  \leq \pi \\
	\pi & \text{ otherwise }
    \end{cases}.
\end{equation}

The first inequality is just using that $\arctan x\leq x$ for positive $x$ and the second uses that $\arctan x \leq \pi/2$ for large values. We observe that for $x > 0$ one has 
\begin{equation}
x \leq -\frac{1}{\pi} + \sqrt{1+t^{2} + \frac{1}{\pi^{2}}}  \Leftrightarrow  \frac{2x}{1+t^{2}-x^{2}} \leq \pi
\end{equation}

Denote $$c(t) = -1/\pi + \sqrt{1 + t^{2} + 1/ \pi^{2} }.$$
Then, one estimates that 
\begin{equation}
   \frac{1}{r_{1}+r_{2}} \Delta\left(\frac{1}{2}+it\right) 
    \leq 
 \int_{0}^{ \tfrac{1}{4}c(2 t) } \frac{ (e^{2\pi x} -1)^{-1} 8x}{1+4 t^{2}-16 x^{2}} \diff x  +  \int_{\tfrac{1}{4}c(2t)}^{\infty} \frac{\pi}{ e^{2 \pi x} - 1} \diff x
\end{equation}

We estimate that $e^{2\pi x} - 1 \geq x$ for positive $x$. The first integral is then exactly solvable and we can estimate it as 
\begin{equation}
    \leq C \frac{ \ln( 2 + |t|)}{\sqrt{1+t^{2}}} \text{ for some }  C>0
\end{equation}
 and for the second we observe that $c(t) \geq C_1$ for some constant $C_{1} > 0$ therefore we can assume $(e^{2\pi x} - 1)^{-1} \leq C_{2} e^{-2 \pi x}$ for some $C_2$. Hence the second we get a bound of
\begin{equation}
    \leq C_{3} e^{-C_{4} t} \text{  for some  } C_{3},C_{4} > 0.
\end{equation}

Combined, this gives us that for all values of $r_{1},r_{2}$
\begin{equation}
    \Delta(\tfrac{1}{2}+it) \leq  (r_{1}+r_{2}) \tfrac{c \ln(2 + |t|)}{\sqrt{1+t^{2}}} \text{ for some } c>0.
\end{equation}
For our use, we can completely ignore the dependence on $t$ and write that 
\begin{equation}
    \Delta(\tfrac{1}{2}+it) \leq  (r_{1}+r_{2}) {c}\text{ for some } c>0.
\end{equation}

By using Stirling's approximation from~\cref{eq:stirling}, we get 
\begin{align}
    & \log |G(s)| 
    =  \Re( \log G(s)) \\
    = & \left( r_{1} \Re(\tfrac{s}{2} \log {s}) -   \tfrac{r_{1}}{2}(\log 2)\Re({s} )  \right) 
    + r_{2} \Re (s \log s) 
    - \left(\tfrac{d}{2} \Re(s \log s) + \tfrac{d}{2} (\log \tfrac{d}{2})\Re(s )  \right) \\ 
      & - \left(\tfrac{r_{1}+r_{2}-1}{2}\right) \Re(\log s) 
      + \tfrac{r_{1}}{2} \log 2 +  \tfrac{1}{2}\log (\tfrac{d}{2})  + \tfrac{r_{1}+r_{2} - 1 }{2} \log (2 \pi) 
      -(r_{1}+r_{2}-1)\Re(s)
 +
      \Delta(s) .
\end{align}
The terms involving $\Re(s \log s)$ cancel out.
We use the fact that $\Re(s) = 1/2$ which gives us $\Re(\log (s)) = \tfrac{1}{2}\log (t^{2} + 1/4)$. 
Then, using our estimate above
we get that 
\begin{align}
     \log |G(s)|  
     \leq &   
     \left( -\tfrac{1}{4} d \log d + c_{1} d \right)  - \tfrac{r_{1}+r_{2}-1}{4}\log\left( t^{2} + \tfrac{1}{4}\right),
\end{align}
where $c,c_{1} > 0$ are some constants.
\end{proof}

\begin{remark}
  When $r_{1}=0$, one can get a less sophisticated bound by employing
	the following classical identities along with $\Gamma(s+1)=s \cdot \Gamma(s)$:
	\begin{align}
	  |\Gamma(\tfrac{1}{2}+it)|^{2}
	  = \frac{2\pi}{e^{\pi t} + e^{-\pi t}}, \\
	  |\Gamma(1+it)|^{2}
	  = \frac{2\pi t}{e^{\pi t} - e^{-\pi t}}.
	\end{align}
	Unfortunately such nice expressions are not available for $\Gamma(\tfrac{1}{4} + it)$ and we have to do the above calculations when $r_{1}>0$.
\end{remark}

\begin{remark}
\label{re:sigma_not_half}
In general, from the proof we also get that for $ \sigma \geq \frac{1}{2}$, for some constant $C(r)$ we have 
\begin{equation}
\log |G(\sigma + it)|
\leq  -\tfrac{r_{1}+r_{2}-1}{4} \cdot \log (t^{2} + \sigma^{2}) +  C(r_{1},r_{2}),
\end{equation}
where $C(r_{1},r_{2})>0$ is some constant that depends on $r_{1},r_{2}$. Hence, the decay along the $t$-axis for a fixed $\sigma \geq 1/2$ is always of the order $ \sim t ^ {-\frac{r_{1}+r_{2}-1}{2}}$.
\end{remark}

\subsection{Estimates on Mellin transform}
We can use Lemma \ref{le:gamma_estimate}
to prove the following corollary.

\begin{corollary}
\label{co:asymptots_on_mellin}
 Let $g = \ind_{[0,R]}$ be the indicator function of an interval. There exists a positive constant $C>0$ such that for all $(r_{1},r_{2}) \in(\mathbb{Z}_{\geq 0})^{2} \setminus \{(0,0)\}$ and for all $t \in \mathbb{R}$ the following inequality holds.
\begin{equation}
	    \left| 
	    \mell{\ind_{[0,R]}}\left( d \cdot ( \tfrac{1}{2} + it )\right) 
	    \left(\frac{ \Gamma(\tfrac{1}{4} + i \tfrac{t}{2})^{r_{1}} \Gamma({\frac{1}{2} + it})^{r_{2}} }{ 2^{\left({\frac{1}{2} + it}\right) r_{2}}\Gamma \left( \tfrac{d}{2}\left(\frac{1}{2} + it\right)\right)}\right) 
	    \right|
\leq 
\frac{R^{\frac{d}{2}} e^{-\frac{1}{4}d\log d+ C d }}{\left(|t|+ 1 \right)^{\frac{r_{1}+r_{2}-2}{2} }}.
 \end{equation}
 Here $d = r_{1}+2r_{2}$.
\end{corollary}
\begin{proof}
We observe that 
\begin{align}
	\left| \mell{\ind_{[0,R]}}\left( d \cdot (\tfrac{1}{2} + it )\right)\right|
	& =  \left| \frac{ R^{ \frac{d}{2} + i  td } }{ \frac{d}{2} + i t d }\right| =
\frac{R^{\frac{d}{2}}}{ \frac{d}{2} \cdot \sqrt{1 + 4t^{2}}} .
\end{align}
Then plugging in Lemma \ref{le:gamma_estimate} does the job. Everything that grows at most exponentially in $d$ can get absorbed in the ${ \exp ( -  \frac{1}{4} d \log d + C d )}$ factor up to readjusting $C>0$.
\end{proof}

\subsection{Hecke integration formula with the indicator function}

\label{ss:hif_ind}

The following is the main tool used in the proof of \cref{th:main}. We write the error term as an exact contour integral for now. The noteworthy matter here is that contour shifting is now possible because of the estimates on gamma functions.
\begin{corollary}
\label{th:hecke_indicatorfunction}
Let $\Lambda \in \Ar(K)$ be as in Theorem \ref{th:hecke_original}.
Let $K$ be a number field of signature $(r_{1},r_{2})$.
Let $B_{R} \subseteq 
\KR$ be an origin-centered ball of radius $R$
with respect to the trace form in (\ref{eq:norm}). 
Then, assuming~\cref{hy:lind} for exponents $\eta_{0},\eta_{1},\eta_{2}$ as in the~\cref{eq:lind}
and assuming that $r_{1} \eta_{1} +2 r_{2} \eta_{2} > 2$, one has
\begin{equation}
  \int_{\Lambda \in \Ar(K)} \card \left( B_{R}\cap \Lambda \right) \diff \Lambda = 
  1+ \vol(B_R)
  + \varepsilon(R,K).
  \end{equation}
  where 
  \begin{equation}
	  \varepsilon(R,K) =
	  \frac{  (2\pi)^{r_2}}{ \vol(\Ar(K)) \pi i } \int_{\frac{1}{2} - i \infty}^{\frac{1}{2} + i \infty}  \Delta_K^{\frac{s}{2}}\,\zeta_{K}(s)\,\frac{R^{s d}}{s d} \left( \frac{ \Gamma( \frac{s}{2})^{r_{1}} \Gamma(s)^{r_2}}{ 2^{sr_2}\, \Gamma(\frac{sd}{2})}\right)\diff s.
	  \label{eq:integran}
\end{equation}
Here the $\tfrac{1}{2}$ may be replaced by any other real in the interval $( \frac{1}{2} , 1 )$.
\end{corollary}
\begin{proof}
	We want to evaluate the value of 
	\begin{equation}
		\int_{\Lambda \in \Ar(K)} \left(  \sum_{v \in \Lambda} \ind_{[0,R]}(\|v\|) \right) \diff \Lambda.
	\end{equation}
	Note that this is not sensitive to whether $\ind_{[0,R]}(R)$ is $1$ or $\tfrac{1}{2}$ since the lattices where this number might differ are of zero measure in $\Ar(K)$. 
	If the absolute value of the integrand of~\cref{eq:integran} as a function of $s= \tfrac{1}{2} + it$ has an asymptotic growth rate of $|t|^{-(1 + \epsilon)}$ for some $\epsilon> 0$, 
	it permits us to use the contour shift argument of Corollary \ref{co:contour_shifting}.

	We observe that because of~\cref{hy:lind}, we know that $|\zeta_{K}(\tfrac{1}{2} + it)| \ll |t|^{ r_{1} (\frac{1}{2} - \eta_{1}) + 2r_{2} (\frac{1}{4} - \eta_{2}) }$ with $\eta_{1},\eta_{2} > 0$ as $|t| \rightarrow  \infty$.
	On the other hand, we know that because of Remark \ref{re:sigma_not_half} and Corollary \ref{co:asymptots_on_mellin}, 
	the rest of the terms in~\cref{eq:integran} grow with an asymptotic rate of $|t|^{-(\frac{r_{1}+r_{2}-1}{2})}$. 
	So $\eta_{1} r_{1} + 2 \eta_{2} r_{2} > 2$ is a sufficient condition to ensure that contour shifting could happen.
\end{proof}

\begin{remark}
  Under the best possible value of $\eta_{2} = \tfrac{1}{4} - \varepsilon$ for some $\varepsilon > 0$ under the Generalized Riemann Hypothesis, one knows that~\cref{co:contour_shifting} holds cyclotomic fields of degree $5$ or more.
\end{remark}

\section{Proof of \texorpdfstring{\cref{th:general}}{the general theorem}}
\label{se:proof}

\begin{proof}
	{\bf (of \cref{th:general})}
    
	Let $R$ be defined as the radius of a ball of volume $V$ in $K \otimes \mathbb{R}$. 
We have $d= r_1 + 2r_2$.
	We can solve for $R$ using
\begin{equation}
	V = \frac{\pi^{\frac{1}{2}  d}}{   \Gamma( 1 + \tfrac{1}{2}d) } R^{d} \Leftrightarrow  R^{ \frac{d}{2}}  
	= \frac{\sqrt{\Gamma(1+\tfrac{1}{2} d )}}{ \pi^{\frac{1}{4}d } } \sqrt{V}.
	\label{eq:accomodate}
\end{equation}
Then Stirling's approximations shows the following dependence between $R$ and $V$ for all $d \in \mathbb{Z}_{\geq 1}$
\begin{equation}
	R^{ \frac{d}{2}} \leq  
	e^{ \frac{1}{4} d  \log d  + c \cdot  d} \sqrt{V}  
 {\quad \text{ for some }} c      >0.
\end{equation}

From Corollary~\ref{th:hecke_indicatorfunction}, we want to show an upper bound on the error term $\varepsilon(V,K) = \varepsilon(R,K)$.
Using Lemma~\ref{le:residue_of_dz}, we can write 
\begin{align}
	  \varepsilon(R,K)
	  & =  \frac{  (2\pi)^{r_2}}{ \vol(\Ar(K)) \pi i } 
	  \int_{\frac{1}{2} - i \infty}^{\frac{1}{2} + i \infty}  
		  \Delta_K^{\frac{s}{2}}\zeta_{K}(s) \frac{R^{s d}}{s d} \frac{G(s)}{2^{sr_2}}\diff s.\\
		  & =  \frac{  (2\pi)^{r_2}}{  \sqrt{\Delta_K} \Res_{s=1} \zeta_{K}(s) \pi i } \int_{- \infty}^{ \infty}  \Delta_K^{\frac{1}{4} + i\frac{t}{2}}\zeta_{K}(\tfrac{1}{2}+it) \frac{R^{ \frac{d}{2}(1+2it)}}{{d} \cdot (\tfrac{1}{2}+it)} \frac{G(\tfrac{1}{2}+it)}{2^{r_2(\frac{1}{2}+it)}}\diff t.\\
		  & =  
		  \frac{ \Delta_K^{-\frac{1}{4}}  (2\pi)^{r_2} 2^{-\frac{1}{2}r_2} R^{\frac{d}{2}}}{   {d} \cdot \pi i \cdot \Res_{s=1} \zeta_{K}(s)  } \int_{- \infty}^{ \infty}  \Delta_K^{ i\frac{t}{2}}\zeta_{K}(\tfrac{1}{2}+it) \frac{R^{it d }}{(\frac{1}{2}+it)} \frac{G(\tfrac{1}{2}+it)}{2^{r_2it}}\diff t .\\
\end{align}

Observe that for any $\epsilon>0$, $| \Res_{s=1} \zeta_K(s) |^{-1}$ is bounded by $C\cdot \Delta_{K}^{\epsilon}$ 
for some $ C > 0$.
in accordance with
Hypothesis \ref{hy:residue_estimate}.
Therefore, we can write that 
\begin{align}
	|\varepsilon(R,K)|
	& \leq 
	C 
	\Delta_{K}^{-\frac{1}{4} + \epsilon}   
	R^{\frac{d}{2}}
	\int_{-\infty}^{ + \infty} 
 \left| \frac{\zeta_K(\tfrac{1}{2} + it)}{\frac{1}{2}+it}
 G(\tfrac{1}{2}+it) \right| \diff t  \\
	& \leq 
 e^{ \frac{1}{4}d \log d  + cd }
	\Delta_K^{-\frac{1}{4} + \epsilon}
	\sqrt{V} \cdot
	\int_{-\infty}^{ + \infty} 
 \left| \frac{\zeta_K(\tfrac{1}{2} + it)}{\frac{1}{2}+it} G(\tfrac{1}{2}+it) \right| \diff t  .
 \label{eq:to_bound}
\end{align}
Here we may have to readjust $c$ to accommodate the constant $C>0$.

We know from Corollary \ref{co:asymptots_on_mellin} that for some $c'>0$
\begin{equation}
 \left| \frac{G(\tfrac{1}{2}+it)}{\frac{1}{2}+it} \right|
 \leq 
\frac{ e^{- \frac{1}{4} d  \log d + c' d  } }{ (|t|+1)^{ \frac{1}{4} d +  \frac{1}{2} }}. \end{equation}

Hence we can write using Hypothesis \ref{hy:lind} that up to readjusting $c'$
\begin{align}
	\int_{-\infty}^{ + \infty} 
 \left| \frac{\zeta_K(\tfrac{1}{2} + it)}{ \frac{1}{2}+it} G(\tfrac{1}{2}+it) \right| dt
& \leq
\Delta_K^{\frac{1}{4} - \eta_0 + \epsilon}
e^{{-\frac{1}{4} d \log d  + c' d}}
\int_{-\infty}^{\infty}
\frac{(1+|t|)^{  \left(\frac{1}{2}- \eta_1 \right)r_{1} +  ( \frac{1}{4}-\eta_{2} )2 r_{2}} }{ (1+|t|)^{\frac{1}{2}(r_{1}+r_{2}-1) }} \diff t ,\\
& \leq
\Delta_K^{\frac{1}{4} - \eta_0 + \epsilon}
e^{{ -\frac{1}{4} d \log d  + c_{1}d}} , \text{ for some } c_{1} > 0.
\label{eq:bound}
\end{align}
since $\eta_{1}r_{1} + 2 \eta_2 r_{2} > 2$. 
This estimate follows from substituting (\ref{eq:bound}) in (\ref{eq:to_bound}):
	\begin{equation}
	 \label{eq:using_hy_disc_final}
    \varepsilon(R,K)\leq\sqrt{V}\cdot\Delta_K^{-\eta_0 + \epsilon}\cdot e^{ c d}, \text{ for some } c> 0.     
	\end{equation}

Finally,~\cref{hy:disc} implies 
that there for any $c > 0$ and any $\epsilon > 0$
$\Delta_{K}^{- \epsilon} e^{c d } \leq C(\epsilon)$ for some constant $C(\epsilon)$ that does not depend on $K \in \mathcal{S}$.
So we can write that 
$$\varepsilon(R,K)\leq C(\epsilon) \sqrt{V}\cdot\Delta_K^{-\eta_0 +  2 \epsilon}, \text{ for some } C(\epsilon)> 0. $$
Now the statement of Theorem~\ref{th:main} follows for any positive $\eta<\eta_0$ by taking $\epsilon > 0$ to be small enough.
\end{proof}

\section{Properties of cyclotomic fields}

\label{se:cyclo_satisfy}

As promised before, we will derive~\cref{hy:lind},
\ref{hy:disc} and \ref{hy:residue_estimate} for the case of $K$ being a cyclotomic number field.
\begin{proposition}
The statements of~\cref{hy:lind},~\cref{hy:disc} and~\cref{hy:residue_estimate} are true for $K$ varying in the set of cyclotomic number fields.
\end{proposition}
\begin{proof}
\begin{enumerate}

    \item 
\subsubsection*{Subconvexity}

Since $r_{1}=0$ for cyclotomic fields, $\eta_{1}$ does not matter.
The following theorem
from \cite{PY23} is the best possible known result toward the required subconvexity
according to the knowledge of the authors.
It shows that $\eta_0,\eta_2$ arbitrarily close to $\tfrac{1}{12}$ are possible. 
\begin{theorem}
\label{th:subconvex_dirichlet} 
{\bf (Petrow-Young, 2023)}
Let $L(s,\chi)$ be the Dirichlet $L$-function of primitive character $\chi$ with conductor $q=q(\chi)$. 
For any $\varepsilon > 0$, there exists a universal constant $C(\varepsilon) > 0$ independent of $q$ or $\chi$
such that for every $\chi$ and every $t \in \mathbb{R}$ one has 
\begin{equation}
|L(\tfrac{1}{2}+it,\chi) | \leq C(\varepsilon) \cdot q^{\frac{1}{6} + \varepsilon}(|t|+1)^{\frac{1}{6}+ \varepsilon} .
\label{eq:subconvexity_dirichlet}
\end{equation}
\end{theorem}

Recall the following equalities for an Abelian number field $K$ should give us a subconvexity bound for $\zeta_K$.
\begin{align}
\zeta_K(s) = \prod_{\chi: \Gal(K) \rightarrow \mathbb{C}^{\times}} L(s,\chi), \\
|\Delta_K| = \prod_{\chi: \Gal(K) \rightarrow \mathbb{C}^{\times}} q(\chi).
\end{align}
So to get the estimate in Hypothesis \ref{hy:lind}, we just multiply 
Equation (\ref{eq:subconvexity_dirichlet}) across all characters of $\Gal(K)$.

\item  \subsubsection*{Lower bound on the discriminant}

This follows from the following explicit formula for the cyclotomic number field \cite{W12}.
\begin{equation}
| \Delta_K |^{\frac{1}{\varphi(n)}} = \frac{n}{  \prod_{p \mid n} p^{\frac{1}{p-1}} }.
\end{equation}
Let $\nu_p(n)$ be the $p$-adic valuation of $n$.
Then, the left hand side of \cref{hy:disc} follows immediately whereas the right side follows by observing
\begin{align}
	\log n - \sum_{p \mid n} \frac{\log p}{p-1} 
	& =  \sum_{p \mid n }
	\left(\nu_{p}(n) - \frac{1}{p-1}\right) \log p.
\end{align}
We observe that 
\begin{equation}
  \nu_p(n) - \frac{1}{p-1} \geq
  \begin{cases}
  \frac{1}{2} \nu_p(n) & \text{if } p \neq 2 \\
 \frac{1}{2} \nu_p(n) - 1 &  \text{if } p =2  ,
  \end{cases}
\end{equation}
  which gives the lower bounds 
  \begin{equation}
	\sum_{p \mid n}\left(\nu_{p}(n) - \frac{1}{p-1}\right) \log p \geq 
  \frac{1}{2}\log n  - \log 2 \gg \log n.
  \end{equation}
  Hence, we establish that 
  \begin{equation}
	  \frac{\log \Delta_{K}}{ \deg K} \rightarrow  \infty
  \end{equation}
  as required.

\item 
\subsubsection*{Lower bound on residues of Dedekind zeta function}

Since cyclotomic number fields are normal (Galois, in particular), one can use the Brauer-Siegel theorem to justify that 
~\cref{hy:residue_estimate} is true. In fact, one can make more precise asymptotics due to the effective version of the Brauer-Siegel theorem due to Stark~\cite{S1974}.

\begin{theorem}
\label{th:stark}
{\bf (Stark, 1974)}

Let $K$ be a normal extension of $\mathbb{Q}$. Then
\begin{equation}
	\label{eq:stark_equation}
  \Res_{s=1} \zeta_K(s) \geq  \frac{c_1}{\deg K\cdot |\Delta_K|^{\frac{1}{\deg K}}} \text{ for some }c_1 > 0.
\end{equation}
\end{theorem}
In \cite{GL22}, the authors report that $c_1 = 0.001448$ seems to work for Equation (\ref{eq:stark_equation}). 
We then simply combine \cref{th:stark} with the estimate in \cref{hy:disc} to get Hypothesis \ref{hy:residue_estimate} 
 using the bound
  $\Res_{s=1}\zeta_K(s)\geq c_1\Delta_K^{-c/\deg K}$
  for any fixed $c>2$ and all sufficiently large $K$.

\end{enumerate}

Since the hypotheses of~\cref{th:general} are satisfied by cyclotomic number fields, we can conclude the following.

\begin{corollary}
   ~\cref{th:main} holds.
\end{corollary}

\end{proof}

\bibliographystyle{amsalpha} 
\bibliography{auth}
  
\bigskip
  \footnotesize

  N.~Gargava, \textsc{Université Paris-Saclay, France}\par\nopagebreak
  \textit{E-mail address}:  \texttt{nihar.gargava@universite-paris-saclay.fr}

  \medskip

  M.~Viazovska, \textsc{École Polytechnique Fédérale de Lausanne,
    Vaud, Switzerland}\par\nopagebreak
  \textit{E-mail address}:  \texttt{maryna.viazovska@epfl.ch}

\end{document}